\documentclass[11pt, a4paper]{article}
\usepackage[utf8]{inputenc}
\usepackage{amsmath, amssymb, amsthm}
\usepackage{graphicx,caption,cite,float}
\usepackage{geometry}
\usepackage{physics}
\usepackage{hyperref}

\title{Integrability of Billiards Inside Cones as a Discrete-Time Hamiltonian System}
\author{Andrey E. Mironov and Siyao Yin}

\date{}

\newtheorem{theorem}{Theorem}

\newtheorem{lemma}{Lemma}

\newtheorem{corollary}{Corollary}

\begin{document}

\maketitle
\let\thefootnote\relax\footnotetext{The work is supported by the Mathematical Center in Akademgorodok under the agreement No.~075-15-2025-348 with the Ministry of Science and Higher Education of the Russian Federation.}

\begin{abstract}
    In this paper, we continue to study billiards inside  cones $K\subset \mathbb{R}^n$ over strictly convex closed $C^3$ manifolds with non-degenerate second fundamental form. Recently we proved that the billiard is superintegrable, i.e., the billiard admits first integrals whose values uniquely determine all billiard trajectories. In this paper we prove that this billiard system admits $n-1$ independent first integrals in involution. Consequently, the system is completely integrable as a discrete-time Hamiltonian system.
  This provides an example of an integrable billiard where the billiard table is neither a quadric nor consists of pieces of quadrics.
\end{abstract}

\section{Introduction and Main Results}

The Birkhoff billiard is a dynamical system that studies the motion of a free particle with unit velocity $\|v\|=1$ inside a domain $U \subset \mathbb{R}^n$ with piecewise smooth boundary $\Gamma = \partial U$. 
Upon reaching the boundary, the particle reflects according to the law of geometric optics: the angle of reflection equals the angle of incidence.
A \textit{caustic} is a hypersurface such that any line tangent to it remains tangent after reflection at the boundary. 
For a planar billiard inside an ellipse $x_1^2/a^2 + x_2^2/b^2 = 1$, the domain $U$ is foliated by confocal caustics and the system admits a first integral 
$$
F = b^2 v_1^2  +   a^2 v_2^2 - (x_1 v_2 - x_2 v_1)^2.
$$
This property is central to the {Birkhoff--Poritsky conjecture}, which states that if a neighborhood of the boundary of a strictly convex smooth billiard table is foliated by caustics, the table must be an ellipse. 
While the conjecture remains open, various partial results have been established \cite{ASK,Bia,BM1,BM2,Bol,G,Kaloshin,KS,Koval,Koz,KT,Tab}. 
Notably, for a piecewise smooth billiard table, if the boundary consists of arcs of confocal quadrics (such as confocal ellipses and hyperbolas) joined together, the system still possesses a first integral and remains integrable.

In higher dimensions ($n \geq 3$), it has been established that if a $C^2$-smooth hypersurface and its caustic both have non-degenerate second fundamental forms, then the boundary must be a piece of a quadric \cite{Ber}. For the global case, Gruber \cite{Gru1} proved that among compact convex billiard tables, only ellipsoids possess caustics.
It is shown that the billiard inside an ellipsoid 
$$
\sum_{i=1}^n \frac{x_i^2}{a_i} = 1, \quad 0 < a_1 < \dots < a_n,
$$
is a completely integrable discrete-time Hamiltonian system \cite{MV, Ves}. 
Specifically, the system admits $n-1$ independent first integrals in involution $F_1, \dots, F_{n-1}$,
\begin{equation}\label{eq:ellipsoid}
    F_k(x, v) = v_k^2 + \sum_{j \neq k} \frac{(x_k v_j - x_j v_k)^2}{a_k - a_j}, \quad k = 1, \dots, n,
\end{equation}
where $\sum_{k=1}^n F_k = \|v\|^2 = 1$.
 Through a change of variables, the billiard map of an ellipsoid is equivalent to the discrete Neumann system on a sphere \cite{Ves}. For fixed values of the first integrals, the trajectories are simultaneously tangent to $n-1$ confocal quadrics; the phase space reduces to invariant Liouville tori, on which the map acts as a translation.

Our previous work \cite{MY1, MY2} considers the billiard inside a convex cone $K \subset \mathbb{R}^n$ defined by
$$
K = \{ tp \mid p \in \gamma, t > 0 \},
$$ 
where $\gamma \subset \{ x \in \mathbb{R}^n \mid x_n = 1 \}$ is a smooth $(n-2)$-dimensional submanifold.  In contrast to Berger's theorem \cite{Ber} and Gruber's theorem \cite{Gru1}, we discovered that such a cone admits a family of spherical caustics centered at the vertex (see Fig. \ref{fig:caustic}), corresponding to the quadratic first integral $I = \sum_{1 \leq i<j\leq n} (x_i v_j - x_j v_i)^2$. 
\begin{figure}[htbp]
    \begin{center}
    \includegraphics[scale=0.22]{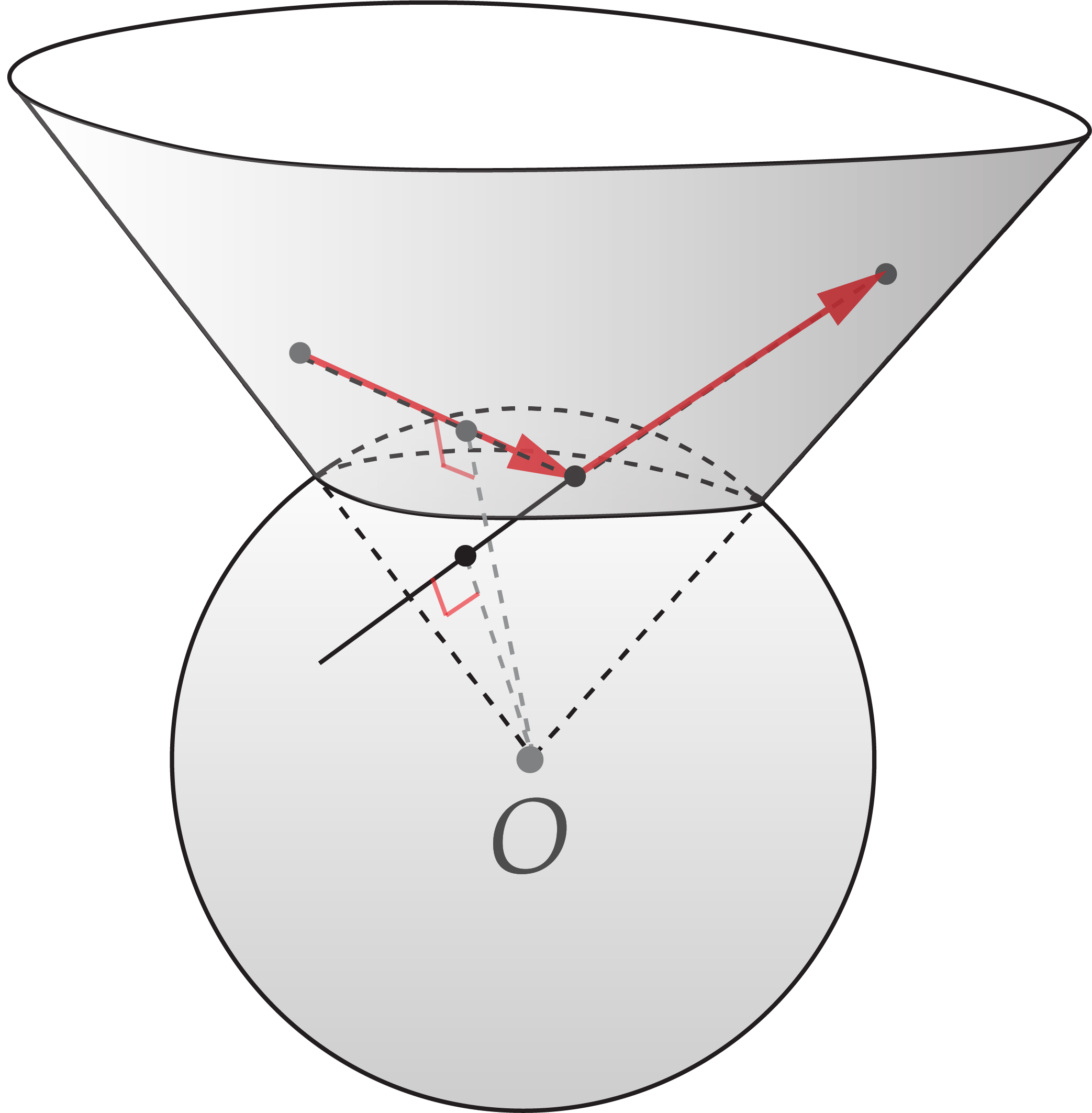}
    \end{center}
    \caption{The sphere as a caustic of the billiard inside a cone.}
    \label{fig:caustic}
  \end{figure}
If $\gamma$ is strictly convex with a non-degenerate second fundamental form, the dynamics depend significantly on the smoothness of $\gamma$: for $C^3$ cases, every trajectory has finitely many reflections, whereas for $C^2$ cases, a trajectory can undergo infinitely many reflections in finite time. In the former case, the system is shown to be superintegrable, admitting $2n-1$ first integrals that uniquely determine each trajectory.

In the present paper, we investigate the integrability of this system as a discrete-time Hamiltonian system. Throughout the following, we assume $\gamma$ is $C^3$-smooth, strictly convex, and possesses a non-degenerate second fundamental form.

To formulate our main result, we first recall the geometry of the phase space of the billiard inside $K$ (see \cite{MY1} for details).
The space of oriented lines in $\mathbb{R}^n$ is identified with the tangent bundle $TS^{n-1}$ of the unit sphere. An oriented line $l$ corresponds to a pair $(v, Q) \in TS^{n-1}$, where $v \in S^{n-1}$ is its direction and $Q \in l$ is the point closest to the origin ($\langle Q, v \rangle = 0$). 
\begin{figure}[htbp]
	\centering
	\includegraphics[width=0.67\linewidth]{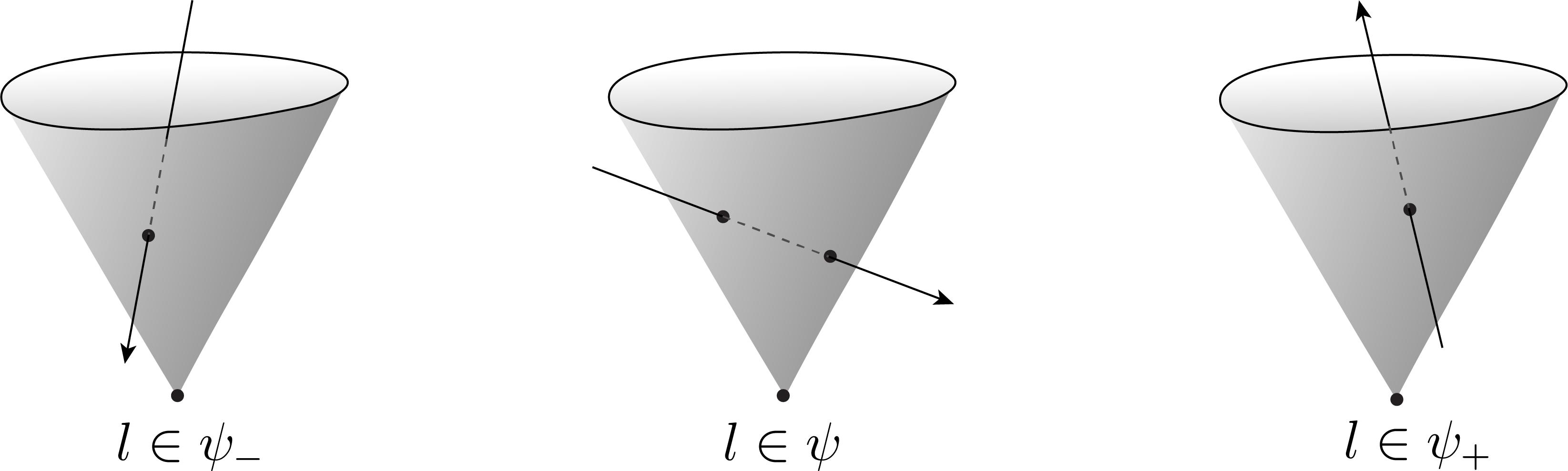}
  \caption{Oriented lines in $\psi_-$, $\psi$, and $\psi_+$. }
  \label{fig:cones-inte-s}
\end{figure}
The phase space $\Psi \subset TS^{n-1}$ is the open set of oriented lines transversally intersecting $K$, which admits a natural partition $\Psi = \psi_- \cup \psi \cup \psi_+$, where (see Fig. \ref{fig:cones-inte-s}):
\begin{itemize}
    \item $\psi_-$ consists of lines approaching from infinity and intersecting the cone exactly once;
    \item $\psi$ consists of lines intersecting the cone twice;
    \item $\psi_+$ consists of lines escaping to infinity after a single intersection.
\end{itemize}

The phase space $\Psi$ carries a natural symplectic structure inherited from the standard symplectic form $\Omega = \sum_{i=1}^n dx^i \wedge dv^i$ on $T\mathbb{R}^n \cong \mathbb{R}^{2n}$. Specifically, $\Psi$ can be identified with an open subset of the symplectic submanifold 
$$
\Psi_0 = \{ (x,v) \in \mathbb{R}^{2n} \mid \phi_1 = \|v\|^2-1=0, \, \phi_2 = \langle x, v \rangle = 0 \}
\cong TS^{n-1} \subset \mathbb{R}^{2n}.
$$
Let $\omega := \Omega|_{\Psi_0}$ be the induced symplectic form on $\Psi_0$ (and its restriction to any open subset of $\Psi_0$).
Let 
$\{\cdot, \cdot\}_{\omega}$ denote 
the Poisson bracket on $\Psi_0$ with respect to $\omega$.
For any smooth functions $f, g$ defined in a neighborhood of $\Psi_0$, the bracket $\{\cdot, \cdot\}_{\omega}$ is given in ambient coordinates by (see Lemma~\ref{lem:1})
\begin{equation}\label{eq:dirac_intro}
    \{f, g\}_{\omega} = \{f, g\} - \frac{1}{2} \left( \{f, \phi_1\} \{\phi_2, g\} - \{f, \phi_2\} \{\phi_1, g\} \right).
\end{equation}
where 
$$
\{f_1, f_2\} = \sum_{i=1}^n \left( \frac{\partial f_1}{\partial x^i}\frac{\partial f_2}{\partial v^i} - \frac{\partial f_1}{\partial v^i}\frac{\partial f_2}{\partial x^i} \right).
$$
is the Poisson bracket on $\mathbb{R}^{2n}$ with respect to $\Omega$.
In particular, for any first integrals $F$ and $G$ of the billiard system expressed in ambient coordinates, their involution with respect to $\{\cdot, \cdot\}_{\omega}$ is equivalent to that under the standard bracket $\{\cdot, \cdot\}$ (see Corollary~\ref{lem:equiv}).

The billiard map $\mu: \psi_- \cup \psi \to \psi \cup \psi_+$, which maps an incident oriented line to a reflected oriented line, is a symplectomorphism with respect to $\omega$ (Lemma \ref{lem:mu-symplectomorphism}).
Following \cite{MY1}, we define the extended billiard map 
$$
\mathcal{M}: \Psi \to \Psi
$$
 by setting 
 $$
 \mathcal{M}|_{\psi_-\cup\psi} = \mu \quad \text{and} 
 \quad 
 \mathcal{M}|_{\psi_+} = \text{id}.
 $$
 The first integrals are constructed via the ``lifting'' procedure. 
 According to Theorem~2 in \cite{MY1}, every trajectory inside $K$ undergoes a finite number of reflections. 
 For any $x \in \Psi$, let 
 $
 m(x) := \min \{ k \geq 0 \mid \mathcal{M}^k(x) \in \psi_+ \}
 $
 be the minimal number of iterates for the trajectory to enter the escaping set $\psi_+$.
  Given a function $f: \psi_+ \to \mathbb{R}$, its \textit{lift} $\tilde{f}: \Psi \to \mathbb{R}$ is defined by 
  \begin{equation}\label{eq:lift-def-intro} 
    \tilde{f}(x) := f(\mathcal{M}^{m(x)}(x)).
  \end{equation}
  Let $\psi_+^{\circ}$ denote the interior of $\psi_+$. 
  While the billiard map $\mathcal{M}$ is discontinuous on the set $\Delta_0 := \psi_+ \setminus \psi_+^{\circ}$ according to 
  Lemma~10 of \cite{MY1},
  the lift $\tilde{f}$ may have discontinuities on the set 
  $$\Delta := \bigcup_{n=0}^{\infty} \mu^{-n}(\Delta_0).
  $$ 
  Let $v^1, \dots, v^{n-1}$ be the coordinate functions on $\psi_+$ representing the first $n-1$ components of 
  the direction vector of the final oriented line of a trajectory.
 As shown in \cite{MY1}, the corresponding lifted functions $\{\tilde{v}^i\}_{i=1}^{n-1}$ are continuous on $\Psi$ and smooth on the open dense subset $\Psi \setminus \Delta$ (see Lemma 12 and Example 2 therein).
 These functions form a subset of the $2n-1$ first integrals established for the superintegrability of the system.
In the present work, we complete the dynamical picture by proving that  $\{\tilde{v}^i\}_{i=1}^{n-1}$ are in involution with respect to the symplectic structure \eqref{eq:dirac_intro}, thereby establishing complete integrability:

\begin{theorem}
    The Birkhoff billiard inside $K \subset \mathbb{R}^n$ is completely integrable as a discrete-time Hamiltonian system. Specifically, the $n-1$ first integrals $\{\tilde{v}^i\}_{i=1}^{n-1}$ obtained via the lifting procedure \eqref{eq:lift-def-intro} of $\{{v}^i\}_{i=1}^{n-1}$ are in involution with respect to the Poisson bracket $\{\cdot, \cdot\}_{\omega}$.
\end{theorem}

Previously, all known examples of completely integrable billiards are quadrics or tables consisting of pieces of quadrics (such as ellipses, ellipsoids, or their confocal families). 
The billiard inside a cone provides the first 
example of 
completely integrable billiard that does not fall into this category. The integrability mainly follows from the 
finiteness property of the trajectories.

\section{Proof of Theorem 1}

\begin{lemma}\label{lem:1}
    \begin{itemize}
        \item[1)] The submanifold $\Psi_0 \subset \mathbb{R}^{2n}$ is a symplectic submanifold with respect to $\Omega$.
        \item[2)] For any smooth functions $f, g$ defined on a neighborhood of $\Psi_0$ in $\mathbb{R}^{2n}$, the induced Poisson bracket on $\Psi_0$, denoted by $\{f, g\}_{\omega}$, is given in ambient coordinates by
        \begin{equation}\label{eq:dirac}
                \{f, g\}_{\omega} = \{f, g\} - \frac{1}{2} \left( \{f, \phi_1\} \{\phi_2, g\} - \{f, \phi_2\} \{\phi_1, g\} \right).
        \end{equation}
    \end{itemize}
\end{lemma}
\begin{proof}
  1) By definition, $\Psi_0$ is a symplectic submanifold of $\mathbb{R}^{2n}$, if 
  \begin{equation}\label{eq:def-nondeg}
      T_z\Psi_0 \cap (T_z\Psi_0 )^{\Omega} = \{0\}
  \end{equation} 
  for all $z \in \Psi_0$,
where $(T_z\Psi_0 )^{\Omega}$ denotes the $\Omega$-orthogonal complement. 
  To prove this, it suffices to show that the $2\times 2$ matrix $C$, defined by 
  $$
  C_{ab} = \{\phi_a, \phi_b\},
  $$
 is invertible on $\Psi_0$.

 A direct computation yields
$$
\{\phi_1, \phi_2\} = \sum_{i=1}^n \left( \frac{\partial \phi_1}{\partial x^i}\frac{\partial \phi_2}{\partial v^i} - \frac{\partial \phi_1}{\partial v^i}\frac{\partial \phi_2}{\partial x^i} \right) = \sum_{i=1}^n \big( 0 \cdot x^i - 2v^i \cdot v^i \big) = -2\|v\|^2.
$$
           Restricted to $\Psi_0$, we have $\|v\|^2 = 1$, hence $\{\phi_1, \phi_2\}|_{\Psi_0} = -2 \neq 0$. The matrix 
           $$
           C = \begin{pmatrix} 0 & -2 \\ 2 & 0 \end{pmatrix}
           $$
           is invertible, confirming that $\Psi_0$ is a symplectic submanifold.

           For completeness, we briefly recall why the invertibility of $C$ guarantees \eqref{eq:def-nondeg}. 
    Since $\Psi_0$ is defined by the independent constraints $\phi_1 = 0$ and $\phi_2 = 0$, 
    we have
    \begin{equation}\label{eq:tangent}
        T_z\Psi_0 = \{ Y \in T_z\mathbb{R}^{2n} \mid d\phi_1(Y) = 0, \, d\phi_2(Y) = 0 \}. 
    \end{equation}           
    Let $X_{\phi_1}$, $X_{\phi_2}$ be the Hamiltonian vector fields associated with $\phi_1$ and $\phi_2$ respectively, defined by
    \begin{equation}\label{eq:Ham-def}
        \Omega(X_{\phi_i}, Y) = d\phi_i(Y)
        \quad \text{~for~all~} Y \in T_z\mathbb{R}^{2n},
        \quad i=1,2.
    \end{equation} 
    Combining \eqref{eq:tangent} and \eqref{eq:Ham-def}, it follows that
        $$
        \Omega(X_{\phi_i}, Y) = 0, \quad 
        \text{for~all~} Y\in  T_z\Psi_0, 
        \quad i=1,2.
        $$ 
        Therefore
        $$ 
        \text{span}\{X_{\phi_1}, X_{\phi_2}\} \subseteq (T_z\Psi_0)^{\Omega}. 
        $$
        The non-degeneracy of $\Omega$ implies $\dim(T_z\Psi_0) + \dim((T_z\Psi_0)^{\Omega}) = 2n$. 
        Hence
       $\dim((T_z\Psi_0)^{\Omega}) = 2$. 
        Since $X_{\phi_1}$ and $X_{\phi_2}$ are linearly independent, we know   
        $$ 
        \text{span}\{X_{\phi_1}, X_{\phi_2}\} = 
        (T_z\Psi_0)^{\Omega}. 
        $$
    Now, let 
    $$
    W = c_1 X_{\phi_1} + c_2 X_{\phi_2}
     $$
    be an arbitrary vector in the intersection $T_z\Psi_0 \cap (T_z\Psi_0 )^{\Omega}$.  
    From 
    $$
    d\phi_i(W) = 0, \quad i=1,2,
    $$
    it follows that 
    \begin{equation}\label{eq:system}
     \begin{cases}
         c_1 d\phi_1(X_{\phi_1}) + c_2 d\phi_1(X_{\phi_2}) = 0 \\
         c_1 d\phi_2(X_{\phi_1}) + c_2 d\phi_2(X_{\phi_2}) = 0
     \end{cases}.
 \end{equation}
 Using 
 $d\phi_i(X_{\phi_j}) = \{\phi_j, \phi_i\}$, the system \eqref{eq:system} is $C \begin{pmatrix} c_1 \\ c_2 \end{pmatrix} = 0$. Since $C$ is invertible, the only solution is $c_1 = c_2 = 0$, implying $W = 0$. This establishes \eqref{eq:def-nondeg} and completes the proof.

    2)
    The induced Poisson bracket on $\Psi_0$ is given by the Dirac's bracket formula (see, e.g., \cite[Section 8.5]{MaR}):
    $$
    \{f, g\}_{\omega} = \{f, g\} - \sum_{a,b=1}^2 \{f, \phi_a\} (C^{-1})_{ab} \{\phi_b, g\},
    $$
    where 
    $$
    C^{-1} 
    =\begin{pmatrix} 0 & 1/2 \\ -1/2 & 0 \end{pmatrix}.
    $$ 
    By substituting the inverse matrix $C^{-1}$ into the Dirac's bracket formula, we obtain \eqref{eq:dirac}.

    Lemma~\ref{lem:1} is proved.
\end{proof}

Since the phase space $\Psi$ is an open subset of $\Psi_0$, it follows immediately from Lemma~\ref{lem:1} that:

\begin{corollary}\label{cor:1}
    $\Psi$ is a symplectic manifold with the symplectic form $\omega := \Omega|_{\Psi}$. The corresponding Poisson bracket, denoted by $\{\cdot, \cdot\}_{\omega}$, is given by the expression in \eqref{eq:dirac}.
\end{corollary}

In the billiard literature, first integrals are conventionally expressed in the ambient coordinates of $\mathbb{R}^{2n}$, for example, the integrals for the ellipsoid \eqref{eq:ellipsoid}.
For any such integral $f$ defined in ambient coordinates,
$$
\{f, \phi_1\} = 2 \sum_{i=1}^n v_i \frac{\partial f}{\partial x_i} = 2 \langle \nabla_x f, v \rangle =0 ,
$$
since $f$ is constant along the straight-line trajectories between reflections. Hence, the correction term in \eqref{eq:dirac} vanishes, yielding the following.

\begin{corollary} \label{lem:equiv}
    Let $f_1, f_2$ be first integrals of the billiard system, expressed as smooth functions in the ambient coordinates of $\mathbb{R}^{2n}$. 
   Then, the induced Poisson bracket $\{f_1, f_2\}_\omega$ on $\Psi$ coincides with the standard Poisson bracket $\{f_1, f_2\}$ on $\mathbb{R}^{2n}$
    $$
    \{f_1, f_2\}_{\omega} = \{f_1, f_2\}.
    $$
    In particular, $f_1$ and $f_2$ are in involution with respect to $\omega$ if and only if they are in involution with respect to the standard symplectic structure $\Omega$.
\end{corollary}

The following lemma establishes that the billiard map of cone $K$ preserves the symplectic structure; this is a general property that holds for Birkhoff billiards defined inside any convex smooth hypersurface.

\begin{lemma} \label{lem:mu-symplectomorphism}
    The billiard map $\mu: \psi_- \cup \psi \to \psi_+ \cup \psi$ is a symplectomorphism, i.e., $\mu^* \omega = \omega$.
\end{lemma}

\begin{proof}
    Assume $K \subset \mathbb{R}^n$ is locally defined as the zero locus $F^{-1}(0)$ of a smooth function $F$ 
  with $\nabla F \neq 0$.    
    Consider an arbitrary point $(x, v) \in \psi_- \cup \psi$ and its image $(x', v') = \mu(x, v)$ under the billiard map $\mu$. 
    We recall that $x$ (resp. $x'$) represents the point on the oriented line closest to the origin, i.e., $\langle x, v \rangle = 0$ (resp. $\langle x', v' \rangle = 0$).
    Let $q \in K$ be the corresponding reflection point on $K$. Then we have the relations
    \begin{equation}\label{eq:q-v-t}
    q = x + tv = x' + t'v',
\end{equation}
    \begin{equation}\label{eq:v-vp}
        v - v' = \alpha \nabla F(q), \quad \alpha \neq 0,
    \end{equation}
  where $t, t'$ are smooth functions on $\psi_- \cup \psi$.

   Let $(q^1,\dots,q^n)$ be the coordinates of $q$ in $\mathbb{R}^n$. From \eqref{eq:q-v-t} we have
   $$
   q^i = x^i + tv^i
   $$ 
   taking the exterior derivative of which yields 
   $$
   dq^i = dx^i + t\,dv^i + v^i\,dt.
   $$ 
   We then have
    $$
    \sum_{i=1}^n dq^i \wedge dv^i = \sum_{i=1}^n dx^i \wedge dv^i + t \sum_{i=1}^n dv^i \wedge dv^i + 
    dt \wedge \left(\sum_{i=1}^n v^i dv^i\right) .
   $$
    Since $\sum dv^i \wedge dv^i = 0$ by antisymmetry and $\sum v^i dv^i = \frac{1}{2} d(\|v\|^2) = 0$ on $\Psi$, the last two terms vanish, leaving
    \begin{equation}\label{eq:dvdq1_final}
    \sum_{i=1}^n dq^i \wedge dv^i = \sum_{i=1}^n dx^i \wedge dv^i.
    \end{equation}
    Similarly,  we have:
    \begin{equation}\label{eq:dvdq2}
    \sum_{i=1}^n dq^i \wedge dv'^i = \sum_{i=1}^n dx'^i \wedge dv'^i.
    \end{equation}

    From \eqref{eq:v-vp}, 
 it follows
    $$
    dv^i - dv'^i = \alpha d \left(\frac{\partial F(q)}{\partial q^i}\right) + \frac{\partial F(q)}{\partial q^i} d \alpha. 
    $$
    Taking the wedge product with $dq^i$ and summing over $i$, we obtain
    $$
    \sum_{i=1}^n d q^i  \wedge \left(dv^i - dv'^i \right) = \alpha \sum_{i,j=1}^n \frac{\partial^2 F(q)}{\partial q^i \partial q^j} dq^i \wedge dq^j +
    \left(\sum_{i=1}^n \frac{\partial F(q)}{\partial q^i} dq^i\right)
    \wedge
    d\alpha.
    $$
    The first term on the right-hand side vanishes because the Hessian $\frac{\partial^2 F}{\partial q^j \partial q^i}$ is symmetric while $dq^j \wedge dq^i$ is antisymmetric. 
    The second term vanishes because $F(q) \equiv 0$ for 
    $q\in K$, which implies $\sum_{i=1}^n \frac{\partial F}{\partial q^i} dq^i = 0$.
    Thus
    \begin{equation}\label{eq:dvdq3}
    \sum_{i=1}^n dq^i \wedge dv^i - \sum_{i=1}^n dq'^i \wedge dv'^i = 0.
\end{equation}
    Combining \eqref{eq:dvdq1_final}, \eqref{eq:dvdq2}, and \eqref{eq:dvdq3}, it follows that
    $$
    \sum_{i=1}^n dv^i \wedge dx^i = \sum_{i=1}^n dv'^i \wedge dx'^i,
    $$
    which is exactly $\mu^* \omega = \omega$. 

    Lemma \ref{lem:mu-symplectomorphism} is proved.
\end{proof}

From  Lemma~\ref{lem:mu-symplectomorphism} it follows:

\begin{lemma} \label{lem:4}
    The extended billiard map $\mathcal{M}: \Psi \to \Psi$ preserves the symplectic structure on the open dense subset $\Psi \setminus \Delta_0 = (\psi_- \cup \psi) \cup \psi_+^\circ$, where $\Delta_0$ denotes the set of discontinuities of $\mathcal{M}$. That is, the relation 
    $$
    \mathcal{M}^* \omega = \omega
    $$
    holds on $\Psi \setminus \Delta_0$.
\end{lemma}

\begin{proof}
    The subset $\Psi \setminus \Delta_0$ is the union of two disjoint open subsets: $\psi_- \cup \psi$ and $\psi_+^\circ$. We examine the pull-back of the symplectic form on each subset:
    
    On $\psi_- \cup \psi$, the map $\mathcal{M}$ is defined by the billiard map $\mu$. By Lemma \ref{lem:mu-symplectomorphism}, $\mu$ is a symplectomorphism, which implies $\mathcal{M}^* \omega = \omega$.
    
    On $\psi_+^\circ$, the map $\mathcal{M}$ is the identity map. Thus, it satisfies $\mathcal{M}^* \omega = \omega$.

    Lemma \ref{lem:4} is proved.
\end{proof}

Recall that by $v=(v^1,\dots,v^n)$ we denote the ending vector of the billiard trajectory, 
where the $v^j$ are functions on $\psi_+$, and by $\tilde{v}^j$ we denote the lifted functions on $\Psi$ (see Introduction). 
As we have mentioned previously, the $\tilde{v}^j$ are continuous first integrals on $\Psi$ and smooth on $\Psi \setminus \Delta$. 
We have the following lemma.

\begin{lemma}\label{lem:2}
    The coordinate functions satisfy $\{v^i, v^j\}_{\omega} = 0$ for all $i, j = 1, \dots, n$.
\end{lemma}

\begin{proof}
    We first observe that $\{v^i, v^j\} = 0$ in the ambient space. Furthermore, for $i=1, \dots, n$,
    $$
    \{v^i, \phi_1\} = \{v^i, \sum_{k=1}^n (v^k)^2 - 1\} = 0.
    $$
    Applying formula \eqref{eq:dirac} to $f=v^i$ and $g=v^j$:
    $$
    \{v^i, v^j\}_{\omega} = \{v^i, v^j\} - \frac{1}{2} \left( \{v^i, \phi_1\} \{\phi_2, v^j\} - \{v^i, \phi_2\} \{\phi_1, v^j\} \right)=0.
    $$
  
    Lemma~\ref{lem:2} is proved.
\end{proof}

From  Lemma~\ref{lem:2} it follows:

\begin{lemma} \label{lem:vi-poisson-commute}
    The first integrals $\tilde{v}^1, \dots, \tilde{v}^n$ are in involution on $\Psi$, i.e.,
$$
        \{\tilde{v}^i, \tilde{v}^j\}_{\omega} = 0, \quad \forall i, j = 1, \dots, n.
$$
\end{lemma}

\begin{proof}
    Consider a point $x \in \Psi \setminus \Delta$. Since the function $m(x)$ is locally constant on $\Psi \setminus \Delta$, there exists a neighborhood $U$ of $x$ on which the map $\mathcal{M}^{m(x)}: U \to \mathcal{M}^{m(x)}(U)\subset \psi_+$ is a smooth symplectomorphism. By definition \eqref{eq:lift-def-intro}, the lifted functions on $U$ are given by the pull-back 
    $$
    \tilde{v}^i = (\mathcal{M}^{m(x)})^* v^i.
    $$
    Hence we have
  $$
        \{\tilde{v}^i, \tilde{v}^j\}_{\omega} = \{v^i \circ \mathcal{M}^{m(x)}, v^j \circ \mathcal{M}^{m(x)}\}_{\omega} = \{v^i, v^j\}_{\omega} \circ \mathcal{M}^{m(x)}.
$$
    According to Lemma \ref{lem:2}, the coordinate functions $v^i, v^j$ satisfy $\{v^i, v^j\}_{\omega} = 0$ on the phase space $\Psi$. It follows that 
    $$
    \{\tilde{v}^i, \tilde{v}^j\}_{\omega} = 0 \circ \mathcal{M}^{m(x)} = 0.
    $$
    This establishes that $\{\tilde{v}^i, \tilde{v}^j\}_{\omega} = 0$ on $\Psi \setminus \Delta$.

    Since the set $\Psi \setminus \Delta$ is open and dense in $\Psi$, we can extend the value of the Poisson bracket $\{\tilde{v}^i, \tilde{v}^j\}_\omega$ continuously to the entire phase space $\Psi$ by setting it identically to zero. Thus, while the first integrals $\tilde{v}^i$ are not continuously differentiable on the singular set $\Delta$, their Poisson brackets are well-defined and vanish continuously on all of $\Psi$. This completes the proof.
\end{proof}

Since the velocity vectors satisfy the unit constraint $\|v\|^2 = 1$, the $n$ first integrals $\tilde{v}^1, \dots, \tilde{v}^n$ satisfy one relation
$$
\sum_{i=1}^n (\tilde{v}^i)^2 = 1.
$$
To establish the complete integrability, we select a subset of $n-1$ integrals and verify their independence.

\begin{lemma} \label{lem:independence}
    The first integrals $\tilde{v}^1, \dots, \tilde{v}^{n-1}$ are functionally independent on $\Psi$.
\end{lemma}

\begin{proof}
    First, the coordinate functions $v^1, \dots, v^{n-1}$ are functionally independent on the subset $\psi_+^{\circ}$ (this is proved in Lemma 14 of \cite{MY1}).
   
    For any $x \in \Psi \setminus \Delta$, let $U$ be a neighborhood where $m(x)$ is constant. Since the map $\mathcal{M}^{m(x)}: U \to \mathcal{M}^{m(x)}(U)\subset \psi_+^{{\circ}}$ is a local symplectomorphism, its pull-back $(\mathcal{M}^{m(x)})^*$ preserves the linear independence of differentials. It follows from
    $$
    d\tilde{v}^i = (\mathcal{M}^{m(x)})^* (dv^i), \quad i=1, \dots, n-1,
    $$
    that $\{d\tilde{v}^1, \dots, d\tilde{v}^{n-1}\}$ are linearly independent on $U$. As $\Psi \setminus \Delta$ is open and dense, the $n-1$ first integrals are functionally independent almost everywhere.

    Lemma \ref{lem:independence} is proved.
\end{proof}

Combining Lemma \ref{lem:vi-poisson-commute} and Lemma \ref{lem:independence}, we conclude that the billiard system admits $n-1$ independent first integrals in involution. This completes the proof of Theorem 1.

\vspace{1cm}

\noindent
{Andrey E. Mironov}\\
Sobolev Institute of Mathematics, Novosibirsk, Russia\\ 
Novosibirsk State University, Novosibirsk, Russia\\ 
Email: \texttt{mironov@math.nsc.ru}

\medskip
\noindent
{Siyao Yin}\\
Sobolev Institute of Mathematics, Novosibirsk, Russia\\
Email: \texttt{siyao.yin@math.nsc.ru}
\vspace{1cm}

\end{document}